\documentclass[12pt,a4paper,reqno]{amsart}
\usepackage{amsmath,amssymb, amsbsy}
\usepackage{color,psfrag}
\usepackage[dvips]{graphicx}
\usepackage{color}
\usepackage{mathtools}
\usepackage{hyperref}

\usepackage{colordvi}
\usepackage{amsmath,amsfonts,amsthm}
\usepackage{verbatim}
\usepackage{amscd}

\usepackage{url}

\usepackage{setspace}

\theoremstyle{plain}

\newtheorem{theorem}{Theorem}[section]

\newtheorem{lemma}[theorem]{Lemma}
\newtheorem{corollary}[theorem]{Corollary}
\newtheorem{remark}[theorem]{Remark}
\newtheorem{definition}[theorem]{Definition}
\newtheorem{problem}[theorem]{Problem}

\newcommand{\Gree}[4]{\mbox{$G^{#1}_{#2}(#3,#4)$}}
\newcommand{\kP}{k_{P}^{M}}
\newcommand{\Gw}{\Omega}
\newcommand{\Real}{\mathbb{R}}

\newcommand{\R}{\mathbb{R}}

\numberwithin{equation}{section}

\newcommand{\Hmm}[1]{\leavevmode{\marginpar{\tiny%
			$\hbox to 0mm{\hspace*{-0.5mm}$\leftarrow$\hss}%
			\vcenter{\vrule depth 0.1mm height 0.1mm width \the\marginparwidth}%
			\hbox to
			0mm{\hss$\rightarrow$\hspace*{-0.5mm}}$\\\relax\raggedright #1}}}

\def\ga{\alpha}            
             
                         \def\vge{\varepsilon}
       \def\vgf{\varphi}    
            \def\gl{\lambda}
\def\gm{\mu}                 
    \def\gr{\rho}

     \def\Gd{\Delta}

\def\Gw{\Omega}              

\newcommand{\dx}{\,\mathrm{d}x}
\newcommand{\dy}{\,\mathrm{d}y}
\newcommand{\dz}{\,\mathrm{d}z}
\newcommand{\dt}{\,\mathrm{d}t}

\numberwithin{equation}{section} \allowdisplaybreaks

\usepackage[text={6in,8.6in},centering]{geometry}
\begin{document}
	\title[Stochastic completeness, and $L^1$-Liouville property]{Stochastic completeness and $L^1$-Liouville property for second-order elliptic operators}
	\author[Debdip Ganguly]{Debdip Ganguly}
	\address{\hbox{\parbox{5.7in}{\medskip\noindent{Department of Mathematics,\\
					Indian Institute of Technology Delhi,\\
					IIT Campus, Hauz Khas, Delhi,\\
					New Delhi 110016, India. \\[3pt]
					\em{E-mail address: }{\tt 
						debdipmath@gmail.com}}}}}
\author[Yehuda Pinchover]{Yehuda Pinchover}
\address{\hbox{\parbox{5.7in}{\medskip\noindent{Department of Mathematics, \\
 Technion - Israel Institute of
Technology, \\
  Haifa 32000, Israel. \\[3pt]
\em{E-mail address : }{ \tt pincho@technion.ac.il}}}}}
\author[Prasun Roychowdhury]{Prasun Roychowdhury}
\address{\hbox{\parbox{5.7in}{\medskip\noindent{Department of Mathematics,\\
					Indian Institute of Science Education and Research,\\
					Dr. Homi Bhabha Road, Pashan,\\
					Pune 411008, India. \\[3pt]
					\em{E-mail address: }{\tt prasunroychowdhury1994@gmail.com}}}}}
	\date{\today}
	\begin{abstract}
	Let $P$ be a linear, second-order, elliptic operator with real coefficients defined on a noncompact Riemannian manifold $M$ and satisfies $P1=0$ in $M$. Assume further that 
		$P$ admits a minimal positive Green function in $M$. We prove that there exists a smooth positive function $\rho$ 
	defined on $M$  such that $M$ is stochastically incomplete with respect to the operator $ P_{\rho} := \rho \, P  $,  that is, 
		\[
		\int_{M} k_{P_{\rho}}^{M}(x, y, t) \ {\rm d}y < 1 \qquad \forall (x, t) \in M \times (0, \infty), 
		\]
		  where $k_{P_{\rho}}^{M}$ denotes the minimal positive heat kernel associated with $P_{\rho}$.
		Moreover,  $M$ is $L^1$-Liouville with respect to $P_{\rho}$ if and only if $M$ is $L^1$-Liouville with respect to $P$. In addition, we study the interplay between 
		stochastic completeness and the $L^1$-Liouville property of the skew product of two second-order elliptic operators. 
	\end{abstract}
	\subjclass[2010]{Primary 31C35,  ; Secondary 35J08, 35A08, 35B09, 58G03.}
	\keywords{Green function,  $L^1$-Liouville, optimal Hardy, stochastically incomplete.}
	\maketitle
\section{Introduction}
Let $M$ be a smooth, noncompact, connected Riemannian manifold of dimension $n$. Let $P$ be a linear, second-order, elliptic operator with real coefficients defined on  $M$, and satisfying $P1=0$ in $M$.

Denote by $\mathcal{C}_{P}(M)$ the cone of positive solutions of the equation $Pu=0$ in $M$. The {\em
	generalized principal eigenvalue} of the operator $P$ is defined by
$$ \gl_0 =\gl_0(P,M)
:= \sup\{\gl \in \mathbb{R} \; \mid\; \mathcal{C}_{P-\lambda}(M)\neq
\emptyset\}.$$
We say that $P$ is {\em nonnegative in} $M$ (and denote it by $P\geq 0$ in $M$) if $\lambda_0:= \lambda_0(P, M)\geq 0.$ Since $P1=0$ in $M$, it follows that 
$\gl_0\geq 0$, that is, $P\geq 0$ in $M.$ 

\medskip

Consider the parabolic operator
\begin{equation}\label{parabolic_operator}
	L u:= \frac{\partial u }{\partial t} + Pu \qquad (x, t) \in M \times (0,\infty),
\end{equation}
and let $k^{M}_{P}(x, y, t)$ be the \emph{minimal positive heat kernel} of the parabolic operator $L$ on the manifold $M.$ By definition, for a fixed $y \in M,$ the function 
$(x, t) \mapsto k^{M}_{P}(x, y, t)$ is a \emph{minimal} positive solution of the equation 
\begin{equation} \label{minimal_heat_kernel}
	Lu = 0 \quad \mbox{in} \ M\times (0, \infty),
\end{equation}
subject to the initial data $\delta_y,$ the Dirac distribution at $y \in M$.
\begin{definition}\label{def_crit}{\em
		Suppose that $\gl_0=\lambda_0(P, M)\geq 0$, and let $\kP$ be the corresponding heat kernel. We say that the operator~$P$ is \emph{subcritical} (respectively, \emph{critical}) in $M$ if for some $x \not = y$, (and therefore for any
		$x \not = y$),  $x,y\in M$, we have
		\begin{equation}\label{def.critical}
			\int_0^\infty k_P^{M}(x,y, t)\,\mathrm{d}t<\infty, \quad
			\left(\mbox{respectively, } \int_0^\infty
			k_P^{M}(x,y, t)\,\mathrm{d}t=\infty\right).
		\end{equation}
		If $P$ is subcritical in $M$, then
		\begin{equation}\label{G_k}
			G_P^{M}(x,y):=\int_0^\infty k_P^{M}(x,y, t)\,\mathrm{d}t \qquad x,y \in M
		\end{equation}
		is called the {\em minimal positive Green function}  of the operator $P$ in $M$.
	}
\end{definition} 
Next we introduce the notion of stochastically completeness.
\begin{definition}\label{defi1}
	{\em Let $(M, g)$ be a connected and noncompact Riemannian manifold of dimension $N.$ Let $P$ be an elliptic operator satisfying $P1=0$ in $M$. Then 
		$M$ is said to be \emph{stochastically complete} (respectively, \emph{stochastically incomplete}) with respect to $P$ if for some $(x, t) \in M \times (0, \infty)$
		(and therefore for all $(x, t) \in M \times (0, \infty)$) we have
		\[
		\int_{M} k_{ P}^{M}(x, y, t) \,\mathrm{d}y = 1, \quad \left(\mbox{respectively, }  \int_{M}
		k_P^{M}(x,y, t)\,\mathrm{d}y < 1 \right).
		\]
	}
\end{definition}
\begin{definition}\label{defi2}
	{\em Suppose that the manifold $M$ and the operator $P$ satisfy the above assumptions. We say that the {\em $L^1$-Liouville property holds on $M$} with respect to $P$  (in short, $M$ is $L^1$-Liouville with respect to $P$) if every nonnegative $L^1$-supersolution of the operator $P$ in $M$ is the constant function.
	}
\end{definition}
\begin{remark}\label{rem-1.4}
	{\em Recall that if $u$ is a positive solution of the equation $Pu=0$ in $M$, then the  generalized maximum principle implies that for any $t>0$, either
		$$	\int_{M} k_{ P}^{M}(x, y, t) u(y)\,\mathrm{d}y = u(x) \quad \mbox{or }  
		\int_{M} k_{ P}^{M}(x, y, t) u(y)\,\mathrm{d}y < u(x) \quad \forall x\in M .$$ 
			 Equivalently, for any $\gl>0$ either 
			 $$\gl \int_{M} G_{ P+\gl}^{M}(x, y) u(y) \,\mathrm{d}y = u(x) \quad \mbox{or }  
			 \gl \int_{M} G_{ P+\gl}^{M}(x, y) u(y)\,\mathrm{d}y < u(x) \quad \forall x\in M.$$			 
			 Moreover, $P$ is critical in $M$ if and only if $P$ admits a unique (up to a multiplicative constant) positive supersolution (see \cite{YP2} and references therein).  
			Therefore, if} $P$ is critical in $M$ and $P1 =0$, then $M$ is stochastically complete and $L^1$-Liouville with respect to $P.$
\end{remark}
Throughout the paper, unless otherwise stated,  we  assume that $P$ is subcritical in $M$.  In other words, we assume that $P$ admits a (unique) minimal positive Green function $G_P^{M}$ on $M$. In this paper we are mainly interested in the following question:
\begin{problem} 
	{\em Given an operator $P$ of the form \eqref{P} on $M$, construct an 
	operator $P_{\rho}: = \rho \, P$ for some positive smooth function $\rho$ on $M$  such that $M$ is stochastically incomplete with respect to $ P_{\rho}$, and such that $M$ is  $L^1$-Liouville with respect to  $P_{\rho}$ if and only 
	if $M$ is $L^1$-Liouville with respect to $P$.}
\end{problem}
We provide an affirmative answer to the above question. Our proof rests on certain constructions of positive supersolutions and related Hardy weights \cite{DFP,GP}, and a generalization of the  Omori-Yau maximum principle (see \cite{BB,BPS}).

\medskip

Stochastically completeness and the $L^1$-Liouville property on manifolds have been studied extensively in the recent years. 
We mention a few of them without a claim of completeness (see \cite{BB, GR1,MV,NN,BPS,PRS}),  a more comprehensive reference can be found in \cite{GR2} and references therein.
It has been shown that certain natural geometric assumptions on manifolds ensure the validity of these properties on a complete manifold, e.g., volume growth, bound 
on the curvature, etc. These questions were well understood in case of the Laplace-Beltrami operator $\Delta_g$ using potential theoretic arguments and exploiting some underlying geometric conditions. 
We mention also \cite{PPS}, where the authors study properties like parabolicity, stochastic completeness and $L^1$-Liouville property in a manifold with boundary,  under the
Dirichlet boundary condition. In \cite{DSA} the authors studied the parabolicity of a manifold with respect to the Neumann boundary conditions.  
Moreover, recently stochastic completeness was characterized in \cite{GIM} in terms of certain properties of nonlinear evolution equations of fast diffusion.
Furthermore,  in \cite{YP1} the author constructed an example of an operator $P$ defined on a complete Riemannian manifold $M$, which 
does not admit invariant positive solution at the bottom of the spectrum, providing a counter example to Stroock's conjecture (see also \cite{GR2,RP}). It turns out that the constructed manifold $M$ 
is stochastically incomplete and not $L^1$-Liouville with respect to $P$. The present article is partly motivated also from the construction 
of  the  example in \cite{YP1}, aiming to find a unified approach to construct stochastically incomplete manifold with respect to a second-order elliptic operator.

\medskip

All the aforementioned results are the primary motivation of the study in this paper. Indeed, we study the two notions defined in definitions~\ref{defi1} and \ref{defi2} for general second-order elliptic operators $P$ satisfying $P1=0$ in $M$.   
The rest of the paper is organized as follows. In Section~\ref{sec2}, we state the assumptions on the operator $P$, and formulate our main result in Theorem~\ref{main_theorem}. Section~\ref{sec3} is devoted 
to the proof of a few key lemmas which provide equivalent conditions  for stochastic completeness
 and the validity of the $L^1$-Liouville property for second-order elliptic operator $P$. In Section~\ref{sec4}, 
we prove our main result (Theorem~\ref{main_theorem}). Section~\ref{sec5} is devoted to the study of the $L^1$-Liouville property and the stochastic (in)completeness for skew product 
operators. Finally, in Section~\ref{sec_optimal-Hardy}
we give an alternative proof of Theorem~\ref{main_theorem}.	 
\section{Preliminaries and the main result}\label{sec2}
This section is devoted to the statement of the main result of the paper. Before going further we must introduce some technical assumptions and few definitions.  

\medskip

Let $M$ be a smooth, noncompact, and connected manifold of dimension $n$, where $n \geq 2$. We consider a second-order elliptic operator $P$ defined on $M$ with real coefficients 
which (in any coordinate system
$(U;x_{1},\ldots,x_{n})$) is of the form
\begin{equation} \label{P}
Pu=-\sum_{i,j =1}^n a^{ij}(x)\partial_{i}\partial_{j}u + b(x)\cdot\nabla u. 
\end{equation}
We assume that for every $x\in M$ the matrix $A(x):=\big[a^{ij}(x)\big]$ is symmetric and that the real quadratic form
\begin{equation*}\label{ellip}
\xi \cdot A(x) \xi := \sum_{i,j =1}^n \xi_i a^{ij}(x) \xi_j \qquad \text{for }
\xi \in \Real ^n
\end{equation*}
is positive definite. Throughout the paper it is assumed that the coefficients of $P$ are H\"{o}lder continuous in $M$. By a {\em solution $v$} of the equation $Pu =0$ in $M$, we mean that $v \in C^2(M)$
 and satisfies the equation pointwise. If $P$ is of the form \eqref{P}, we obviously have 
\begin{equation}\label{div}
 P1=  0.
\end{equation}
Whenever we consider the adjoint operator we assume that $P^*$, the formal adjoint operator of $P$, also has H\"{o}lder continuous coefficients. We say that $P$ is {\em symmetric} if $P=P^*$.
\begin{remark}
	{\em It is well known that any smooth manifold may be considered as an embedded manifold in $\R^N$, for some $N\geq n$. Hence, the Euclidean metric and the Lebesgue measure on $\R^N$ 
	induce a Riemannian metric and a measure on $M$, which will be considered throughout the paper in case the given manifold $M$ is not a priori Riemannian. Another possible structure is 
	induced by the principal part of the operator $P$. Namely, the matrix $A=\left(a^{ij}\right)$ induces a Riemannian metric $g_A$ on the manifold $M$ by
$$\left(g(\partial/\partial x_i, \partial/\partial x_j)\right) := \left(a_{ij}\right)= A^{-1}.$$	  
}
\end{remark}
We recall two fundamental properties of the heat kernel $k_P^M$ associated with the operator $P$.
\begin{lemma}\label{properties}
	Let $P$ be a nonnegative second-order elliptic operator defined on $M$. Then the minimal positive heat kernel $k_{P}^{M} (x, y, t)$ satisfies the following properties:
	\begin{enumerate}
		\item $k_{P}^{M} (x, y, t)$  satisfies the Chapman-Kolmogorov equation (the semigroup property):
		\[
		k_{P}^{M} (x, y, s + t) = \int_{M} k_{P}^{M} (x, z, s) k_{P}^{M} (z, y, t) \dz \qquad  \forall s, t > 0 \mbox{ and } \forall \ x, y \in M.
		\]
		\item
		$
		\kP(x, y, t) \geq 0  \qquad \forall t > 0  \mbox{ and } \ \forall  x, y \in M.
		$	
	\end{enumerate}
\end{lemma}
Now we are in a position to state the main theorem of the paper. 
\begin{theorem}\label{main_theorem}
	Let $(M, g)$ be a smooth connected and noncompact manifold of dimension $n \geq 2.$ Let $P$ be a  subcritical operator in $M$ of the form \eqref{P}.
	Assume that $G(x):=G_P^M(x,o)$, the minimal positive Green function with a singularity at $o\in M$, satisfies 
	\begin{equation}\label{eq-G0}
	\underset{x\rightarrow \bar \infty}\lim G(x) \, = \, 0,
	\end{equation}
	where $\bar\infty$ is the ideal point of the one-point compactification of $M$. 
	
	Then there exists a positive smooth function $\rho$ defined on $M$ such that the operator $P_{\rho} := \rho P $ satisfies the following properties: 
	\medskip

	$(1)$  $M$  is stochastically incomplete with respect to $P_{\rho}.$
	
	\medskip 

     $ (2)$    $M$ is $L^1$-Liouville with respect to $P_\rho$ if and only if $M$ is  $L^1$-Liouville with respect to $P.$
\end{theorem}
\begin{remark}
	{\em Consider a 2-dimensional Riemannian manifold $(M,g)$, and consider a conformal change of the metric $g\mapsto \hat g := (\gr(x))^{-1}g$, where $\gr$ is a positive smooth function on $M$. 
	Then the Laplace-Beltrami operator on $(M,\hat g)$ is given by $\Gd_{\hat g}=\gr(x)\Gd_g$. 
	} 
\end{remark}
\section{Key Lemmas}\label{sec3}
In this  section we  prove some of the main ingredients used in the proof of Theorem~\ref{main_theorem}. Recently there have been extensive research on the equivalent conditions of 
stochastic completeness/incompleteness (see \cite{BB, BPS, PPS, PRS}). In these papers it has been shown that stochastic completeness is equivalent to 
Omori-Yau maximum principle. However, all these results are proved  in the particular case of $P := -\Delta_{g}$, where  $\Delta_{g}$ is the  Laplace-Beltrami operator in $M$. It turns out that  one can 
easily extend with some modifications these results to the case of second-order  elliptic operators satisfying  appropriate assumptions. The next lemma is devoted to address this characterization. 
\begin{lemma}\label{lem1}
	Let $M$ be a smooth, connected, noncompact manifold,  and let $P$ be a subcritical operator (not necessarily symmetric) of the form \eqref{P}.
	Then the following assertions are equivalent. 
	\begin{enumerate}
		\item $M$ is stochastically complete with respect to $P$. 
		
		\medskip
		
		\item For every $\lambda < 0,$ the only nonnegative, bounded classical solution $u$ of $Pu = \lambda u$ is $u = 0.$ 
		
		\medskip
		
		\item If for any $T\in (0, \infty),$ the Cauchy problem 
		\begin{align*}
		\begin{cases}
			&\frac{\partial v}{\partial t} + P v = 0,  \\
			& v|_{t = 0^{+}} = 0,
\end{cases}
		\end{align*}
		has a bounded solution in $M \times (0, T)$, 
		then necessarily $v \, = \,0.$
		
		\medskip
		
		\item For  $u \in C^2(M)$ with $\sup_{M} u  < \infty $ and for  $\alpha > 0$, define
		\begin{equation}\label{eq_gwga}
		 \Omega_{\alpha} := \{ x \in M \mid  u(x) > \sup_{M} u - \alpha \}. 
		 \end{equation}
		 Then for every such $u$  and $\ga>0$, we have $\inf_{\Omega_{\alpha}} (-P)u \leq 0.$

		\medskip
		
		\item For every $u \in C^{2}(M)$ with $\sup_{M} u  < \infty,$ there exists a sequence $\{ x_{n} \}$ such that for every $n\in \mathbb{N}$, 
		$$
		 u(x_{n}) \geq \sup_{M} u - \frac{1}{n}\,,  \quad \emph{and} \quad  (-P) u(x_{n}) \leq \frac{1}{n}\,.
		$$
	\end{enumerate}
\end{lemma}
The  above lemma is known for the case when $P : = -\Delta_g.$ A complete proof can be found, for instance in (\cite[Theorem~6.2]{GR2},   \cite[Theorem~1.1]{PRS}), 
see also \cite{GR1}. For the sake of completeness we present the proof below,  with the necessary modifications needed for a general operator $P$ of the form \eqref{P}.
\begin{proof}
	$ (2) \implies (1)$. Suppose that $ (1) $ is not true. Then $M$ is stochastically incomplete with respect to $P$. Define 
	$$u(x, t)=\mathcal{P}_t 1(x,t):= \int_{M} k^{M}_{P}(x, y, t)\dy.$$ 
	By our assumption, $ 0\leq  u(x, t) < 1$.  For any $\lambda >0$  set,
	$$
	v(x) := \int_{0}^{\infty} e^{-\lambda t} u(x, t)  \mbox{d}t = \int_M G_{P+\gl}^M(x,y)\dy. 
	$$
	It can be easily checked (see \cite{YP2}) that  
	\begin{align*}
		(P+\gl)v(x) = 1.
	\end{align*}
	Clearly, $0 < v < 1/\lambda$. Let $w := 1-\lambda v$. Then  $(P+\gl)w=0$ and $0<w<1$, which contradicts $(2)$. 
	
	\medskip 
	
	$ (3) \implies (2)$. Suppose $(2)$ is not true, and let $v$ be a bounded nonzero nonnegative function satisfying $Pv = \lambda v$ with $\gl<0$. We may assume that 
		$0<v \leq 1$. Then $u(x, t) := e^{-\lambda t} v(x)$ solves the Cauchy problem 
	\begin{equation}\label{cauchy0}
		\begin{cases}
		&\frac{\partial u}{\partial t} + P u = 0,  \\
		& u|_{t = 0^{+}} = v(x).
		\end{cases}
	\end{equation}
	Also, $w(x,t):= \int_M k^{M}_{P}(x, y, t)v(y) \dy$ solves the above Cauchy problem. Using the minimality of $w$, we conclude that $0 \leq w \leq 1$. Moreover, we claim $u \neq w$ for $t > 0$. Indeed, since $\int_{M} k^{M}_{P}(x, y, t)\dy\leq 1$, it follows that 
	$$
	\sup_{x \in M} w(x, t) \leq \sup_{x \in M} v(x).
	$$
	On the other hand, for $t > 0$
	$$
	\sup_{x \in M} u(x, t) = e^{-\lambda t} \sup_{x \in M} v(x) > \sup_{x \in M} v(x) \geq \sup_{x \in M} w(x, t),
	$$
	and the claim is proved. Hence, the nonzero nonnegative and bounded function $z:= u - w$ solves the Cauchy problem 
	\begin{align}\label{eq-Cauchy0}
\begin{cases}
		&\frac{\partial z}{\partial t} + P z = 0,  \\
		& z|_{t = 0^{+}} = 0,
		\end{cases}
	\end{align}
which contradicts $(3).$
	\medskip
	
	$(1) \implies (3).$ We also show this by contraposition. Assume  $u(x,t)$ is  a nonzero bounded solution of \eqref{eq-Cauchy0} in $M \times (0, T)$ for some $T > 0$. Without loss of generality, we may assume that $\sup_{x \in M} u(x,t) > 0$ and $\sup_{x \in M} |u(x,t)| < 1$. Then the function $w:=1 - u$ is positive and $\inf_{x \in M} w(x,t) < 1$. Since the function $w$ is a solution to the Cauchy problem
	\begin{align}\label{cauchy1}
	\begin{cases}
		&\frac{\partial w}{\partial t} + P w = 0,  \\
		& w|_{t = 0^{+}} = 1,
		\end{cases}
	\end{align}
	and  $\mathcal{P}_t1:= \int_M K_P^M(x,y,t)\dy$ is the minimal positive solution to \eqref{cauchy1}, we conclude that $\mathcal{P}_t 1 \leq w$.
	Therefore, for some $x\in M$  and $t\in (0,T)$,
	$$
	\mathcal{P}_t 1=\int_M K_P^M(x,y,t)\dy <1,
	$$
	and M is stochastically incomplete. 
	\medskip
	
	$(4) \implies (5)$ and $(5) \implies (2)$ are trivial. Now to complete the proof it remains to show $(2) \implies (4).$  We argue by a contradiction. Assume that there exists a 
	function $u \in C^2(M)$ with $\sup_{M} u  < \infty $ satisfying for some $\alpha > 0$  
	$$
	\inf_{\Omega_{\alpha}} (-Pu) \geq C > 0,
	$$
	where $\Gw_\ga$ is defined in \eqref{eq_gwga}. Then following \cite[Theorem~1.1]{PRS}, we set 
	$$\Omega^* := \left\{ x \in M : (-P) u > \frac{C}{2} \right\}.$$ 
	Obviously, $\overline{\Gw_\ga}\subset \Gw^*$. Moreover,  $u + \alpha - \sup_{M} u$ is a 
	subsolution of $(P + \lambda)$ on $\Gw^*$ with $\lambda = C/(2\alpha)$.  Hence $u_{\alpha} := \mbox{max} \{u + \alpha - \sup_{M} u, 0 \}$ is a subsolution of
	$(P  + \lambda)$ in $M$.  Clearly,  $0 < u_\alpha \leq \alpha.$ Furthermore, any positive constant is a supersolution of $P+\gl$ in $M$,
	 and choosing a constant strictly greater than $\alpha$ we have a supersolution $u^{+} > u_{\alpha}.$
	
	Using the
	``sub/supersolution method" and Perron's method for $P+\gl$, we conclude that there exists a solution $v$ of the equation $(P+\gl)u=0$ in $M$ that satisfies  $0<u_{\alpha} \leq v \leq u^+$ which does not vanish identically, but this contradicts $(2)$.	
\end{proof}
\begin{remark}\label{stoch_incomplete}
	{\em
		In view of Lemma~\ref{lem1},  stochastically incomplete manifold with respect to $P$ is characterized by the existence of a function $u \in C^{2} (M)$ with $ u^{*} := \sup_{M} u < \infty$ such that for any sequence $\{ x_{n} \} \in M$ satisfying
		\begin{equation} \label{Given_hypothesis}
			\lim_{n \rightarrow \infty} u(x_{n}) = u^{*} 
		\end{equation}
		we have 
		$$
		\limsup_{n \rightarrow \infty} (-P) u(x_{n}) > 0.
		$$
	}
\end{remark}
Next, we address the $L^1$-Liouville property. Grigor'yan proved in \cite{GR1, GR2} that the $L^1$-Liouville property with respect to the Laplace-Beltrami operator $\Gd_g$ is equivalent to the non-integrability of the Green function of $\Gd_g.$  Using Grigor'yan's approach  (see also \cite{CC}), we extend in the following lemma the aforementioned result to the case of a subcritical operator of the form \eqref{P}. In fact, Corollary~\ref{def1}, demonstrates that the $L^1$-Liouville property is equivalent to  the non-integrability of the heat kernel of $P$.
\begin{lemma}\label{liouville_lemma}
	Let $(M, g)$ be a smooth, connected, noncompact manifold,  and let $P$ be a subcritical operator of the form \eqref{P}.
	Then the following assertions are equivalent. 
	\begin{enumerate}
		
		\item $M$ is not $L^1$-Liouville with respect to $P$.
		
		
		\item   For any $y \in M$, the minimal Green function $G^{M}_{P}(\cdot,y)$ is in $L^{1}(M)$.
		
	\end{enumerate}
\end{lemma}
\begin{proof}
	$ (2) \implies (1).$ Define 
	$$
	u := \min \{ G^{M}_{P}(\cdot,y), 1\}.
	$$
	Since $P1 = 0,$ it follows that  $u$ is a nonconstant positive $L^{1}(M)$-supersolution.  Hence, $M$ is not $L^1$-Liouville with respect to $P$.
	
	$ (1) \implies (2).$ Since $P1 = 0$, the operator $P$ satisfies the weak and strong maximum principle. Suppose that $u \in L^1(M)$ is a nontrivial nonnegative supersolution of $P$ in $M$. Then by the  strong maximum principle, $u > 0$ in $M$. 
	
	Fix $y \in M$ and a compact smooth exhaustion $\{ M_{j} \}_{j = 1}^{\infty}$ of $M$ such that $K=\overline{B(y, 1)}  \Subset M_1$. Using Harnack's inequality,
	$$
	G^{M}_{P}(x,y) \leq C u(x),  \quad \forall x \in \partial K,
	$$ 
	where $C := C(K)$ is a positive constant. Then using the \emph{generalized maximum principle}, we obtain $ G^{M_j}_{P}(x,y) \leq C u(x),  \quad \forall x \in M_j \setminus K$, 
	where $G^{M_j}_{P}(y,x)$ denotes the Dirichlet Green function in $M_j.$ Now letting $j \rightarrow \infty$,  we obtain 
	$$
	G^{M}_{P}(x,y) \leq C u(x),  \quad \forall x \in  M\setminus K.
	$$
	Therefore,
	\begin{align*}
		\int_{M} G^{M}_{P}(x,y) \dx & =  \int_{K} G^{M}_{P}(x,y)\dx + \int_{M\setminus K} G^{M}_{P}(x,y)\dx \\
		& \leq  \int_{K} G^{M}_{P}(x,y) \dx + C \int_{M\setminus K} u(x) \dx<\infty
	\end{align*}
	since $G^{M}_{P}$ is locally integrable. Hence, $G^{M}_{P}(\cdot,y)$ is in $L^{1}(M)$.
\end{proof}

\medskip

In view of the above lemma, we have. 
\begin{corollary}\label{def1}
	Let $M$ be a smooth, connected, noncompact manifold, and let $P$ be an elliptic operator of the form \eqref{P}. Consider its heat kernel $k_P^M(x, y_,t)$. 
	Then $M$ is $L^1$-Liouville with respect to $P$ if and only if there exists $y_0\in M$ such that $$\int_M\int_0^\infty k_P^M(x, y_0,t)\dt \dx=\infty.$$
\end{corollary} 
\begin{corollary}\label{complete_imply_liouville}
	Let $P$ be a symmetric 
	 operator in $M$ of the form \eqref{P}. If $P$ is  stochastically complete, then $M$ is  $L^1$-Liouville with respect to $P$.
\end{corollary}
\begin{proof}
	First consider the case when $P$ is subcritical,  fix $y\in M$. By  Lemma~\ref{liouville_lemma}, it suffices to show that $G(\cdot,y)$ is not integrable.  Using Tonelli's theorem we get 
	\begin{align*}
		 \int_{M} G_{P}^{M}(x,y)\dx 
		  =\! \int_{M}\int_{0}^{\infty}k_{P}^{M}(x,y, t)  \dt\dx  
		  =\! \int_{0}^{\infty} \int_{M} k_{P}^{M}(x,y, t)  \dx \dt =\!\int_{0}^{\infty}\!\!\dt = \infty.
	\end{align*}

	On the other hand, by Remark~\ref{rem-1.4}, $P$ is critical in $M$ if and only if $P$ admits (up to a multiplicative constant) a unique positive supersolution. Therefore, since $P1=0$, it follows that in the critical case, $M$ is $L^1$-Liouville with respect to $P$.
\end{proof}
We recall an elementary  lemma from \cite[Lemma~2.2]{YP1}.
\begin{lemma}\label{lem2}
	Let $P$ be an operator of the form \eqref{P} and $\rho$ a positive smooth function. Assume that $P$ is subcritical in $M$.  Consider the operator 
		$P_{\rho} = \rho P $ defined on $M$.  Then $P_{\rho}$ 
	is subcritical in $M$ and the minimal positive Green function $G^{M}_{P_{\rho}}$ of $P_{\rho}$ in $M$ is given by 
	\begin{equation}\label{modified_green}
		G^{M}_{P_{\rho}}(x, y) = \frac{G^{M}_{P}(x, y)}{\rho(y)}\, .
	\end{equation} 
\end{lemma}
\section{Proof of the main theorem}\label{sec4}
This section is devoted to the proof of Theorem~\ref{main_theorem}.  The proof is essentially based on finding a ``test function"  $u \in C^{2}(M)$ which satisfies the hypothesis of 
Remark~\ref{stoch_incomplete}.  An alternative method involving optimal Hardy weights is presented in Section~\ref{sec_optimal-Hardy}. 
\begin{proof}[Proof of Theorem~\ref{main_theorem}] 	
	(1) Fix $o\in M$ and a positive  Radon measure $\gm$ on $M$ with a smooth positive density $\gm$.
	We assume that the corresponding {\em Green potential} $G_\gm$ is finite. That is, we assume that for some $x\in M$ (and therefore, for any $x\in M$) 
	\begin{equation}\label{eq_finite_pot}
	G_\gm(x):= \int_{M}  \Gree{M}{P}{x}{y} \gm(y)\dy<\infty.
	\end{equation}
Hence, $PG_\gm=\gm$. Denote by $G(x)=\Gree{M}{P}{x}{o}$. By the minimality of $G$,  it follows that for $\vge>0$ small enough there exists $C_\vge>0$ such that  $G(x)\leq C_\vge G_\gm(x)$ in $M\setminus B(o,\vge)$. 

We further assume that $G_\gm(x)\leq  C_1 G(x)$ in $M$. In other words, we have 
\begin{equation}\label{GasimGvgf}
G_\gm \asymp G \qquad  \mbox{in }  M\setminus B(o,\vge). 
\end{equation} 
In particular,  \eqref{eq-G0} and \eqref{GasimGvgf} imply that $G_\gm$ is bounded in $M$ and  
$$
\lim_{x \rightarrow \bar\infty}  G_{\mu} (x) =0. 
$$
For necessary and sufficient conditions for the validity of \eqref{GasimGvgf} see \cite[Lemma~6.1]{GP} and references therein. 

	Set $\rho:=1/\gm$, and consider the operator $P_{\rho} = \gr P $. Lemma~\ref{lem2} implies that $P_{\rho}$ is subcritical in $M$ with the minimal positive Green function 
	$$\Gree{M}{P_\gr}{x}{y}=\frac{\Gree{M}{P}{x}{y}}{\gr(y)}\,. $$
 
	In view of Remark~\ref{stoch_incomplete}, in order to prove stochastic incompleteness of $P_\gr$ in $M$, it is sufficient to construct a function $u \in C^{2} (M)$ with $ u^{*} := \sup_{M} u < \infty$ such that for any sequence $\{ x_{n} \} \in M$ satisfying $\lim_{n \rightarrow \infty} u(x_{n}) = u^{*}$,	we have 
		$$\limsup_{n \rightarrow \infty} (-P_\gr) u(x_{n}) > 0.$$

Let $u:=-G_\gm$. Then clearly $u<0$ and $u$ satisfies $u(x) \rightarrow 0$ as $x \rightarrow \bar \infty$, where $\bar\infty$ is the ideal point of the one-point compactification of $M$. 	Hence, $ u^{*} := \sup_{M} u =0 < \infty$ and a sequence  
$\{ x_n \} \subset M$ satisfies $\lim_{n\to\infty}u^*(x_n)=u^*=0$ if and only if $x_n \to\bar \infty$. On the other hand, since $(-P_{\rho})u=1$, it follows that for any sequence $\{ x_n \}$ we have
$$\limsup_{n \rightarrow \infty} (-P_{\rho})u(x_{n}) = 1 > 0.$$
Therefore, 	$M$ is stochastic incomplete with respect to $P_\gr$\,.

\medskip

	(2) We need to prove that  $M$ is  $L^1$-Liouville with respect to the operator $P_{\rho} := \rho P $ if and only if $P$ is $L^1$-Liouville with respect to $P$.   In view of Lemma~\ref{liouville_lemma}
	 this is equivalent to the non-integrability of the  Green function $\Gree{M}{P_\gr}{x}{y}$ as a function of $x$.
	Indeed, 
	$$\hspace{2cm} \int_{M}  \Gree{M}{P_\gr}{x}{y}\dx= \int_{M}\  \frac{\Gree{M}{P}{x}{y}}{\gr(y)}\dx= \gm(y) \int_{M} G^{M}_{P}(x, y) \, {\rm d}x. $$
	For fixed $y,$ the conclusion follows immediately. 
	\end{proof}
\begin{remark}{\em
In light of the above proof, it follows that Theorem~\ref{main_theorem} holds true for any $\gr=1/\gm$ such that $\lim_{x \rightarrow \bar\infty}  G_{\mu} (x) =0$ (without assuming \eqref{GasimGvgf}).	
}\end{remark}	
\begin{remark}{\em
 Let $W_\gm:= \gm/G_\gm$. By \cite[Lemma~6.2]{GP}, if \eqref{GasimGvgf} holds true, then the operator $P-W_\gm$ is positive-critical
in $M$ with respect to the Hardy-weight $W_\gm$ and  $G_\gm$ is its ground state.
}\end{remark}	
\section{Skew product operators and  stochastically incompleteness}\label{sec5}
This section is devoted to the study of the skew product of operators $P_i$ of the form \eqref{P}, defined on smooth, noncompact, and connected manifolds $M_i$, $i=1,2$, and its relations to the stochastic completeness and $L^1$-Liouville properties of the operators $P_i$.  Let us recall the definition of
skew product (see \cite{MM1, MM2, MM3}). 
\begin{definition}\label{defi3}{\em 
Let $M = M_1 \times M_2$ be the Cartesian product of two manifolds $M_1$ and $M_2$, and denote a point $x$ in $M$ by $x = (x_1, x_2) \in M = M_1 \times M_2.$ The {\em skew product operator}
 $P$ of operators $P_i$ of the form \eqref{P} defined on $M_i$, $i=1,2$, is given by  
\begin{equation}\label{eq_skew}
P := P_1 \otimes I_2 + I_1 \otimes P_2,
\end{equation}
where $I_i$ is the identity map on $M_i$. }
\end{definition}

Next, we recall a well known lemma on the heat kernel of a skew product operator. For completeness, we provide the proof. 
\begin{lemma}\label{subcritical_skew}
 Let $P_1$  and $P_2$ be nonnegative operators of the form \eqref{P} defined on $M_1$ and $M_2$, respectively. Let $k_{P_1}^{M_1}$ and $k_{P_2}^{M_2}$ be the  minimal positive heat  kernels of $P_1$ and $P_2$, respectively. 
 Then the minimal positive heat kernel $k^M_P$ of the skew product operator $P$  on $M=M_1\times M_2$   satisfies 
 \begin{equation}\label{sph}
		k_P^M(x, y, t) = k_{P_1}^{M_1}(x_1, y_1, t)\:k_{P_2}^{M_2}(x_2, y_2, t), 
 \end{equation}
 where $x = (x_1, x_2),\: y = (y_1, y_2) \in M.$ Furthermore, if at least one of the above operators is subcritical,  
 then $P$ is subcritical in $M$.
 \end{lemma}
\begin{proof}	
The product formula \eqref{sph} follows immediately from the definition of skew product. 
	In order to prove subcriticality, we need to prove that for some $x\neq y$, (and therefore for any $x\neq y$), $x,y\in M$,
	\begin{equation*}
				\int_{0}^{\infty}k_P^M(x,y,t) \dt
		= \int_{0}^{\infty}k_{P_1}^{M_1}(x_1, y_1, t)\:k_{P_2}^{M_2}(x_2, y_2, t) \dt<\infty.
	\end{equation*}
	Without loss of generality, we may assume that $P_2$ is subcritical in $M_2$.
	By Fubini's theorem it is enough to prove that 
	\begin{equation*}
		\int_{M_1}\int_{0}^{\infty}k_{P_1}^{M_1}(x_1, y_1, t)\:k_{P_2}^{M_2}(x_2, y_2, t) \dt\:\dy_1<\infty.
	\end{equation*}
	By the minimality of the heat kernel, it follows that  
	\begin{equation*}
		\int_{M_1}k_{P_1}^{M_1}(x_1,y_1,t)\dy_1 \leq 1 \: \text{ for some (and hence for all) }\  (x_1,t)\in M_1\times(0,\infty).
	\end{equation*}
	Therefore, by Tonelli's theorem and the subcriticality of $P_2$ in $M_2$,  we obtain,
	\begin{align*}
		 \int_{M_1}\int_{0}^{\infty}k_{P_1}^{M_1}(x_1, y_1, t)\:k_{P_2}^{M_2}(x_2, y_2, t)\dt\:\dy_1 
		 \leq 1\cdot\int_{0}^{\infty}k_{P_2}^{M_2}(x_2, y_2, t)\dt <\infty. \qquad \qedhere
	\end{align*}
\end{proof}
The next result leads us to the question of stochastic (in)completeness of a skew product operator. 
\begin{lemma}\label{one_si}
	Let $P_1$  and $P_2$ be operators of the form \eqref{P} defined on $M_1$ and $M_2$, respectively. Then the manifold $M=M_1\times M_2$ is 
	 stochastically incomplete with respect to the skew product operator $P$ if at least one of the  manifolds $M_i$ is stochastically incomplete with respect to $P_i$.
\end{lemma}
\begin{proof}
	Without loss of generality, we may assume that $M_1$ is stochastically incomplete manifold with respect to the operator $P_1$, and hence for all $(x_1, t) \in M_1 \times (0, \infty)$ there holds 
		\begin{equation}\label{EL5.4}
			\int_{M_1} k_{ P_1}^{M_1}(x_1, y_1, t) \,\mathrm{d}y_1 < 1.
		\end{equation}	
	Now for $x = (x_1, x_2),\: y = (y_1, y_2) \in M$, by Tonelli's theorem we have
\begin{align*}
	\int_{M} k_{ P}^{M}(x, y, t) \,\mathrm{d}y&= \left(\int_{M_1}k_{P_1}^{M_1}(x_1, y_1, t)\dy_1\right)\!\!
	\left(\int_{M_2}\:k_{P_2}^{M_2}(x_2, y_2, t)\dy_2\right)\\
	&\leq \int_{M_1}k_{P_1}^{M_1}(x_1, y_1, t)\dy_1 < 1.
\end{align*}
	Another  simple alternative proof can be easily derived  using Remark~\ref{stoch_incomplete}. 
\end{proof}
The following result concerns the interplay between stochastic (in)completeness and 
the $L^1$-Liouville property of a skew product operator of the form \eqref{eq_skew}.
\begin{theorem}\label{thm_skew}
Let $P_1$  and $P_2$ be operators of the form \eqref{P} defined on $M_1$ and $M_2$, respectively. Let $P := P_1 \otimes I_2 + I_1 \otimes P_2$ 
	defined  on $M := M_1 \times M_2$ be the corresponding skew product operator. Then the following assertions hold true: 
	\begin{enumerate} 
		
		\item If at least one of $M_i,$  $i = 1, 2$ is not $L^1$-Liouville with respect to $P_i$, and the other operator is symmetric, then the skew product operator 
		$P$ on $M$ is not $L^1$-Liouville.  
		
		\medskip 
		
		\item If one of the $M_i,$ $i =1, 2$ is $L^1$-Liouville with respect to $P_i$ and the other is symmetric, then the product manifold $M$ is $L^1$-Liouville with respect to $P$. 
		
		\medskip 
		
		\item Assume that both $M_i$,   $i = 1, 2$, are stochastically complete with respect to $P_i$, then the product manifold $M$ is stochastically complete with respect to $P$. 
		Moreover, if $P$ is symmetric in $M$, then $M$ is $L^1$-Liouville with respect to $P$. 
		
		\medskip 
		
		\item  If one of the $M_i,$ $i =1, 2$ is stochastically incomplete with respect to $P_i$, then $M$ is stochastically incomplete with respect to $P$.
	\end{enumerate}
\end{theorem}
\begin{proof}
	$(1)$ Without loss of generality, we assume that $M_1$ is not $L^1$-Liouville with respect to $P_1$ and $P_2$ is symmetric. Then
		\begin{align*}
	& \int_{M} G_{P}^{M}(x_1,x_2,y_1,y_2)\dx_1\dx_2 
	 = \int_{0}^{\infty}\int_{M_1}\int_{M_2}k_{P_1}^{M_1}(x_1, y_1, t)\:k_{P_2}^{M_2}(x_2, y_2, t) \dx_2\dx_1 \dt\\
	&=\int_{0}^{\infty}\int_{M_1}k_{P_1}^{M_1}(x_1, y_1, t) \dx_1 \left(\int_{M_2}k_{P_2}^{M_2}(x_2, y_2, t) \dx_2\right)\dt\\
	&\leq \int_{M_1}\int_{0}^{\infty}k_{P_1}^{M_1}(x_1,y_1,t)\dt\dx_1<\infty.
	\end{align*}
	Hence,  $M$ is  not $L^1$-Liouville with respect to $P$.
	
	\medskip

	$(2)$ Without loss of generality, we assume that $M_1$ is $L^1$-Liouville with respect to $P_1$ and $P_2$ is symmetric. Let $(y_1,y_2)$ be any point on $M_1 \times M_2$ 
	and assume $P$ is not $L^1$-Liouville, i.e., 
	$$\int_{M_1}\int_{M_2}\int_{0}^{\infty}k_{P_1}^{M_1}(x_1, y_1, t)\:k_{P_2}^{M_2}(x_2, y_2, t) \dt \dx_1\dx_2<\infty.$$
	Then using Fubini's theorem and the symmetricity of $P_2$, we conclude,  $$\int_{M_1}\int_{0}^{\infty}k_{P_1}^{M_1}(x_1, y_1, t) \dt \dx_1<\infty,$$
	and hence $M_1$ is not $L^1$-Liouville with respect to $P_1$ which is a contradiction. Therefore, 
	$M$ is $L^1$-Liouville with respect to $P$. 
	
	\medskip
	
	$(3)$  Since  both $M_i$, $i =1, 2$ are stochastically complete with respect to $P_i$, therefore, we have for $i =1, 2,$
	\begin{equation*}
		\int_{M_i}k_{P_i}^{M_i}(x_i,y_i,t)\dy_i = 1 \text{ for any } (x_i,t)\in M_i\times (0,\infty).
	\end{equation*}
	Using Tonelli's theorem we obtain 
	\begin{align*}
		 \int_{M} k_{P}^{M}(x_1,x_2,y_1,y_2,t)\dy_1\dy_2
		= \int_{M_1}\int_{M_2}k_{P_1}^{M_1}(x_1, y_1, t)\:k_{P_2}^{M_2}(x_2, y_2, t) \dy_2\dy_1 = 1.
	\end{align*}
	So, $M$ is stochastically complete with respect to $P$. Furthermore, under the symmetricity assumption, Corollary~\ref{complete_imply_liouville} implies that $P$  is $L^1$-Liouville as well. 

		\medskip 
	
	$(4)$ By Lemma~\ref{one_si}, $M$ is stochastically incomplete with respect to $P$.
\end{proof}

The following result concerns a case where the $L^1$-Liouville property holds for the skew product operator $P$. 
\begin{corollary}
	If one of the $P_i$, $i =1, 2$ is symmetric
		in $M_i$ and the other is critical, then the product manifold $M$ is 
		$L^1$-Liouville with respect to $P$.
\end{corollary}
\begin{proof}
	Since a critical operator of the form \eqref{P} is $L^1$-Liouville, the corollary follows from part $(2)$ of Theorem~\ref{thm_skew}. 
\end{proof}
\begin{remark} {\em 
	We recall that in \cite{GR2,YP1} the authors present a counterexample to a conjecture of A.~Grigor'yan of  a skew product manifold $M\times K$ that does not satisfy  the $L^\infty$-Liouville property, where $(M,g)$ is a complete Riemannian manifold 
	(with $\gl_0(-\Gd_g, M) = 0$) satisfying the $L^\infty$-Liouville property in $(M,g)$, and $K$ is any compact  Riemannian manifold.
}
\end{remark}
\section{Optimal Hardy weight and Stochastically incomplete}\label{sec_optimal-Hardy}
In the present section, we give an alternative proof of Theorem~\ref{main_theorem} under an additional assumption. 
\begin{proof}[Alternative proof of Theorem~\ref{main_theorem}] 	

{\bf Step 1 :}  First we recall the main result of Frass, Devyver and Pinchover \cite[Theorem~4.12]{DFP} (also see \cite{PI}): Let $P$ be a subcritical operator of the form 
\eqref{P} in $M,$ and let $G^M_{P}$ be its minimal positive Green function. Assume that $P$ and $G^M_{P}$ satisfy all the hypotheses of Theorem~\ref{main_theorem}.  Let $0 \leq \varphi \in C_c^{\infty}(M), \vgf\neq 0$, and consider the Green 
potential given by 
$$
G_{\varphi}^M (x) = \int_{M}  G^M_{P} (x, y) \varphi(y) \, {\rm d}y.
$$ 
Then $G_{\varphi}^M$ satisfies 
$$
P(\sqrt{G_{\varphi}^M}) = W \sqrt{G_{\varphi}^M} \geq 0,
$$
where $W := \dfrac{P(\sqrt{G_{\varphi}^M})}{\sqrt{G_{\varphi}^M}}$, and  $W = \dfrac{|\nabla(G^M_{\varphi})|^2_A}{4|G^M_{\varphi}|^2} \quad \mbox{in} \ M \setminus \mathrm{supp}\,  \varphi.$ Moreover, $W$ is an optimal Hardy-weight for $P$ in $M$ (That is, $P-W$ is null-critical in $M$, and in particular, for any $\gl>1$ the operator $P-\gl W \not \geq 0$ outside any compact set in $M$).  

\medskip

{\bf We assume that $W\!>\!0$ outside a compact set $K$ satisfying $\mathrm{supp}\,\vgf \Subset K\subset M$.}

\medskip

{\bf{ Step 2 :}} Let $\rho$ (to be chosen later) be a positive smooth function on $M.$ Consider the operator $P_{\rho} = \rho P.$  In view of Remark~\ref{stoch_incomplete}, 
in order to prove stochastic incompleteness, it is sufficient to construct a function satisfying certain conditions. Let $u$ be a strictly negative $C^2$-smooth function such that  
$$u(x) := - (G_{\varphi}^M(x))^{1/2} \qquad \forall \ x \in M \setminus K.$$
Then 
$$(-P_{\rho}) u =  \rho W(G^{M}_{\varphi})^{1/2} \qquad \forall \ x \in M \setminus K.$$
Moreover, since  $G_{\varphi}^M \asymp G^M_{P}$ in a neighbourhood 
of $\bar\infty$, it follows from our assumption \eqref{eq-G0} that $u(x) \rightarrow 0$ as $x \rightarrow \bar\infty.$ 
Hence, $ u^{*} := \sup_{M} u =0$, and $\lim_{n\to\infty}u(x_n)=u^*$ if and only if  $x_n\rightarrow \bar\infty$. For any such a sequence, we eventually have,
$$
(-P_{\rho}) u(x_{n})  =  \rho(x_n) W(x_n) (G^{M}_{\varphi}(x_n))^{1/2}.
$$
Choose a strictly positive smooth function $\rho$ on $M$ such that 
$$
\rho(x) :=  \left( W(x) (G^{M}_{\varphi}(x))^{1/2} \right)^{-1} \qquad \forall x\in M\setminus K,  
$$
letting $ n \rightarrow \infty,$ we conclude  
$$\limsup_{n \rightarrow \infty} (-P_{\rho})u(x_{n}) = 1 > 0.$$
Hence, by Remark~\ref{stoch_incomplete}, $M$ is stochastically incomplete with respect 
to $ P_{\rho} := \rho P$. 
\end{proof}
\begin{remark}{\em 
We note that  the extra assumption that $W>0$ in a neighbourhood of $\bar \infty$ is satisfied in many cases. For example, if $M$ is a smooth bounded domain in $\R^n$, 
and the coefficients of $P$ are up to the boundary H\"older continuous, then by the Hopf lemma, $W>0$ in a neighbourhood of $\partial M$.
}
\end{remark}

\medspace

   \par\bigskip\noindent
\textbf{Acknowledgments.}
 D. Ganguly is partially supported by INSPIRE faculty fellowship (IFA17-MA98). Y.P.  acknowledges the support of the Israel Science Foundation (grant 637/19) founded by the Israel Academy of Sciences and Humanities. P. Roychowdhury is supported by Council of Scientific \& Industrial Research (File no.  09/936(0182)/2017-EMR-I).

\end{document}